\documentclass{amsart}
\usepackage{amssymb}

\usepackage{xcolor, overpic}
\usepackage[colorlinks=true, pdfstartview=FitV, linkcolor=blue, citecolor=blue, urlcolor=blue]{hyperref}

\theoremstyle{plain}
\newtheorem{thm}{Theorem}[section]
\newtheorem{lemma}[thm]{Lemma}

\newtheorem{cor}[thm]{Corollary}

\newtheorem*{introtheorem}{Main Theorem}

\numberwithin{equation}{section}
\numberwithin{figure}{section}

\theoremstyle{definition}

\newtheorem{example}[thm]{Example}
\newtheorem{remark}[thm]{Remark}

\newtheorem{claim}[thm]{Claim}

\newcommand{\G}{\Gamma}

\begin{document}

\title[Non-quasiconvex subgroups]{Non-quasiconvex subgroups of hyperbolic groups via Stallings-like techniques}

\author{Pallavi Dani}
\address{Department of Mathematics, Louisiana State University, Baton Rouge, LA 70803-4918}
\curraddr{}
\email{pdani@math.lsu.edu}
\thanks{The first author was supported in part by NSF Grant \#DMS-1812061.}

\author{Ivan Levcovitz}
\address{Department of Mathematics, Tufts University, Somerville, MA 02144}
\email{Ivan.Levcovitz@tufts.edu}
\thanks{The second author was supported  in part by a Technion fellowship.}


\date{}

\dedicatory{}


\maketitle

\begin{abstract}
We provide a new method of constructing non-quasiconvex subgroups of hyperbolic groups by utilizing techniques inspired by Stallings' foldings. The hyperbolic groups constructed are in the natural class of right-angled Coxeter groups (RACGs for short) and can be chosen to be $2$-dimensional. More specifically, given a non-quasiconvex subgroup of a (possibly non-hyper\-bolic) RACG, our construction gives a corresponding non-quasiconvex subgroup of a hyperbolic RACG. 
 We use this to construct explicit examples of non-quasiconvex subgroups of hyperbolic RACGs including subgroups whose generators are as short as possible (length two words), finitely generated free subgroups, non-finitely presentable subgroups, and subgroups of fundamental groups of square complexes of nonpositive sectional curvature.
\end{abstract}

\bigskip

\section{Introduction} \label{sec:intro}
Given a group $G$ with generating set $T$, a subgroup $H$ of $G$ is \textit{quasiconvex} (with respect to~$T$) if some finite neighborhood of $H$ in the Cayley graph $C$ of $(G,T)$ contains every geodesic 
that connects a 
pair of vertices in $H \subset C$.  
If $G$ is hyperbolic, then 
whether or not $H$ is quasiconvex
does not depend on the finite generating set chosen.
In this article, we provide a simple construction of finitely generated non-quasiconvex subgroups of hyperbolic groups using techniques inspired by Stallings' foldings. The hyperbolic groups we construct are additionally in the natural class of right-angled Coxeter groups and can be chosen to admit a  geometric action on a CAT(0) square complex.  This adds to the limited list of known techniques for constructing non-quasiconvex subgroups of hyperbolic groups. 

Given a simplicial graph $\Gamma$ with vertex set $V$ and edge set $E$, the corresponding \textit{right-angled Coxeter group} (RACG for short) $W_\Gamma$ is the group with presentation:
\[\langle s \in V ~|~ s^2 = 1 \text{ for all } s \in V, st = ts \text{ for all } (s,t) \in E \rangle\]
We refer to the generators in this presentation as the \textit{standard generating set} for $W_\Gamma$, and quasiconvexity of subgroups of $W_\Gamma$ will always be assumed to be with respect to this generating set.
We say that $W_\Gamma$ is \textit{$2$-dimensional} if $\Gamma$ does not contain a $3$-cycle.
 Such RACGs act geometrically on a CAT(0) cube complex that is at most $2$-dimensional (their Davis complex). 

Let $W_\Gamma$ be a RACG and $H < W_\Gamma$ a subgroup generated by a finite set of words $T$ in $W_\Gamma$. Let $R$ be the ``rose graph'' associated to $T$, i.e., a bouquet of $|T|$ circles 
in which each circle is subdivided and labeled by an element of $T$. 
In~\cite{DL}, we define a completion $\Omega$ of $H$ as the direct limit of a sequence of fold, cube attachment and cube identification operations performed on $R$, and 
we
 show that many properties of $H < W_\Gamma$ are reflected by those of $\Omega$. Notably, $H$ is quasiconvex in $W_\Gamma$ if and only if $\Omega$ is finite.

Given any simplicial graph $\Gamma$, we define in this article the notion of $\Gamma$-partite graphs
	(see Section \ref{sec:partite_graphs}).
These graphs have a large-scale structure which reflects the graphical structure of $\Gamma$. We show, via an explicit construction, that given any such  $\Gamma$, 
 there are infinitely many $\Gamma$-partite graphs $\Delta$, such that $W_\Delta$ is hyperbolic and $2$-dimensional
(see Theorem~\ref{thm:square_free}).

The importance of $\Gamma$-partite graphs in our setting is explained by the following construction.
Given a $\Gamma$-partite graph $\Delta$, and a subgroup $H< W_\Gamma$  generated by a finite set of words $T$ in~$W_\Gamma$, we define a corresponding  finite set of words $\overline T$ in~$W_\Delta$ which generates a subgroup $\overline H < W_\Delta$.  Moreover, fold and cube 
 attachment operations that are performed on the rose graph associated to $T$ have direct analogues 
 to operations performed on the rose graph associated to $\overline T$. In fact, we show there is a finite-to-one
map from the $1$-skeleton of a completion of $\overline H$ to the $1$-skeleton of a completion of $H$ that is defined by simply collapsing certain bigons to edges.
 In particular, we deduce that $H$ is quasiconvex in $W_\Gamma$ if and only if $\overline H$ is quasiconvex in $W_\Delta$. 
Using this construction, we show:
 
\begin{introtheorem}[Theorem \ref{thm:functor} and Corollary \ref{cor:nonquasiconvex}]\label{thm:intro_thm_main}
	Given any finitely generated,
	non-quasiconvex subgroup $H$ of a (possibly non-hyperbolic) RACG $W_\Gamma$, there exist infinitely many $\Gamma$-partite graphs $\Delta$, such that $W_\Delta$ is hyperbolic, $2$-dimensional and contains the 
	finitely generated, non-quasiconvex subgroup $\overline H$. Moreover, if $H$ is torsion-free (resp.~not normal in $W_\Gamma$) then $\overline H$ is torsion-free (resp.~not normal in $W_\Delta$). 
\end{introtheorem}

Non-quasiconvex subgroups of \textit{non-hyperbolic}
 RACGs are plentiful.  They can, for instance, be found in RACGs that split as a product of infinite groups. Consequently, our construction allows for many easy, explicit examples of non-quasiconvex subgroups of hyperbolic RACGs. 

There are other, but not many, known methods of constructing non-quasiconvex subgroups of hyperbolic groups,
 and we give a non-extensive summary of such constructions. 
Many of the known examples come from 
kernels of homomorphisms:
infinite-index, infinite, normal subgroups of hyperbolic groups  are known to always be non-quasiconvex \cite{ABCFLMSS}.
A classical such example is that of the fundamental group of a hyperbolic $3$-manifold fibering over the circle, which contains a normal, non-quasiconvex surface subgroup (see \cite{Thurston} or \cite{otal}). 
 A celebrated early construction of Rips \cite{Rips}
gives 
 non-finitely presentable, non-quasiconvex,
normal subgroups using small cancellation theory.   
 Another source of such kernels is 
 Bestvina-Brady Morse theory developed in \cite{Bestvina-Brady}.
 This  is utilized, for instance, in \cite{Brady} to give examples of non-hyperbolic (therefore non-quasiconvex) finitely presented subgroups of hyperbolic groups. Bestvina-Brady Morse theory is also applied to RACGs in \cite{KNW} giving another source of examples in these groups. 
 
An approach to producing non-quasiconvex examples which are not necessarily normal 
 has been to use HNN extensions.
 For instance, in \cite{Kapovich} this is used to show that any non-elementary, torsion-free, hyperbolic group appears as a non-quasiconvex subgroup of a different hyperbolic group. 
In a similar vein~\cite{mitra, BBD, BDR} use iterated HNN extensions to construct free non-quasiconvex subgroups of hyperbolic groups with varying distortion functions.

\subsection{Applications}
We use our main theorem to construct several explicit examples of nonquasiconvex subgroups in hyperbolic groups with additional desirable properties. 
In a first such 
application, we construct
hyperbolic, $2$-dimensional RACGs, each of which contains a finitely generated, non-quasiconvex subgroup whose generators are all length $2$ words in the standard generating set of the RACG (see Example \ref{ex:2_gen}). 
As any subgroup of a RACG that is generated by length \textit{one} words is convex (and therefore quasiconvex), this construction gives minimal non-quasiconvex subgroups, in the sense that the generators are as short as possible.

In another application
(see Theorem \ref{thm:free_subgroup}),
 we construct
hyperbolic, $2$-dimensional, one-ended RACGs, each of which contains a finitely generated, 
non-quasiconvex
subgroup $F$ that is free.  Furthermore, $F \cap K$ is not normal in $K$ for any 
finite-index subgroup $K$ of the RACG.
By the latter property, these subgroups cannot be obtained, even virtually,
as kernels of any homomorphism 
(which, as discussed earlier, is a known strategy of producing non-quasiconvex examples).

The completions of our subgroups
have an explicit description, being analogues of well-understood, simple completions of subgroups of non-hyperbolic groups.  
This, together with  the fact that the fundamental group of a completion of a  torsion-free subgroup is the subgroup itself, provides an elementary proof that 
the examples from the previous paragraph are free. 

In Theorems \ref{thm:not_fp} and \ref{thm:surface_subgroups} respectively, 
we demonstrate how our method can be used to explicitly construct hyperbolic groups with subgroups that are not finitely presentable and with non-quasiconvex subgroups containing closed surface subgroups.

Wise defines a notion of sectional curvature for $2$-complexes \cite{Wise} and shows that the fundamental group of a compact Euclidean $2$-complex with negative sectional curvature is always locally quasiconvex (i.e., all finitely generated subgroups are quasiconvex) and is consequently coherent (i.e., all finitely generated subgroups are finitely presented).
On the other hand, fundamental groups of $2$-complexes with nonpositive 
sectional 
curvature provide a 
fringe case 
where coherence is conjectured  \cite[Conjecture 12.11]{Wise-survey} and yet local quasiconvexity can fail.
This is witnessed by hyperbolic free-by-cyclic groups, as these groups are non-locally quasiconvex (they contain a normal, infinite, infinite-index subgroup)
and can be realized as  
fundamental groups of $2$-complexes 
of nonpositive sectional curvature by \cite[Theorem~11.3]{Wise}.
We give  
explicit examples (see Theorem~\ref{thm:non_pos}) of this phenomenon by constructing 
non-locally quasiconvex 
hyperbolic RACGs that are virtually fundamental groups of right-angled square complexes of nonpositive sectional curvature.  These examples are likely to be virtually free-by-cyclic.

\subsection*{Acknowledgements}
IL would like to thank Michah Sageev for helpful conversations. The authors thank Jason Behrstock,  Mahan Mj, and Dani Wise and an anonymous referee for comments.   
\section{Background} \label{sec:background}

\subsection{Completions}
A \textit{cube complex} is a cell complex whose cells are Euclidean unit cubes (of any dimension). Given a graph $\Gamma$, a cube complex is \textit{$\Gamma$-labeled} if its edges are labeled by vertices of $\Gamma$. All $\Gamma$-labeled cube complexes considered in this article have the property that edges on opposite sides of squares have the same label. In particular, given the labels of $d$ edges which are  all incident to a common vertex and contained in a common $d$-cube,  the labels of the remaining edges of this $d$-cube are completely determined. This will often be used without mention.

By a \textit{path} in a cube complex $\Sigma$, we 
mean a simplicial path in the $1$-skeleton of $\Sigma$. The label of a path is a word $s_1 \cdots s_m$ in $V(\Gamma)$, such that $s_i$ is the label of the $i$th edge traversed by $\Sigma$.

Given a $\Gamma$-labeled cube complex $\Sigma$, we define three operations, each of which produces a new $\Gamma$-labeled cube complex (see~\cite{DL} for additional details):

\smallskip\noindent
\textbf{Fold operation:} Let $e$ and $f$ be distinct edges of $\Sigma$ which are incident to a common vertex and have the same label. 
A \textit{fold operation} produces a new complex in which $e$ and $f$ are identified to a single edge with the same label as $e$ and $f$. 

\smallskip\noindent
\textbf{Cube identification operation:} Given two or more distinct $d$-cubes in $\Sigma$ with common boundary, a cube identification operation identifies these cubes. 

\smallskip\noindent
\textbf{Cube attachment operation:} Let $e_1, \dots, e_d$, with $d \ge 2$, be edges of $\Sigma$ all incident to a common vertex, 
with distinct labels corresponding
to a $d$-clique of $\Gamma$. A \textit{cube attachment operation} attaches a labeled $d$-cube $c$ to $\Sigma$ by identifying a set of $d$ edges of $c$, all incident to a common vertex, to the edges $e_1, \dots, e_d$. Note that the labels of edges of $c$ are completely determined by the labels of $e_1, \dots, e_d$.

\smallskip\noindent
A $\G$-labeled cube complex is \textit{folded} if no fold or cube identification operations can be performed on it. 
It is \textit{cube-full} if, given any set of $d \ge 2$ edges which are incident to a common vertex and whose labels 
form $d$-clique of $\Gamma$, 
these edges are
contained in a common $d$-cube.

Let $\Sigma$ be a $\Gamma$-labeled cube complex. Let $\Omega_0, \Omega_1, \dots$ be a 
 sequence of cube complexes such that $\Omega_0 = \Sigma$ and, for $i > 0$, 
the complex $\Omega_i$ is obtained from $\Omega_{i-1}$ by performing one of the three above operations. 
There is a natural map $\Omega_i \to \Omega_{i+1}$.
 Let $\Omega$ be the direct limit of this sequence of complexes. We say that $\Omega$ is a \textit{completion} of $\Sigma$ if $\Omega$ is folded and cube-full. 
If $\Sigma$ is finite,
one can always 
construct a completion of $\Sigma$ 
\cite[Proposition~3.3] {DL}. 

Let $H$ be a finitely generated subgroup of a RACG $W_\Gamma$, generated by a set of words $T = \{w_1, \dots, w_m\}$ in $W_\Gamma$. 
The \textit{rose graph} associated to $T$ is the $\Gamma$-labeled graph with a base vertex $B$ and a $|w_i|$-cycle attached to $B$ with label $w_i$ for each $1 \le i \le m$. For this article, it suffices to define a \textit{completion of $H$ with respect to $T$} as a completion $\Omega$ of the rose graph associated to $T$. 
The image of 
$B$
in $\Omega$ is defined to be the base vertex of $\Omega$.

We say that a word $w$ in $V(\Gamma)$ is \textit{reduced} if it has minimal length out of all possible words in $V(\Gamma)$ representing the same element of $W_\Gamma$ as $w$.
We summarize the key properties of completions that we will need:

\begin{thm} \label{thm:completions} \cite{DL}
	Let $H$ be a subgroup of a RACG $W_\Gamma$ generated by a finite set $T$ of words in $W_\Gamma$. Let $\Omega$ be a completion of $H$ with respect to $T$
 with basepoint $B$. Then:
	\begin{enumerate}
		\item \label{thm:completions_qc} The completion $\Omega$ is finite if and only if $H$ is quasiconvex in $W_\Gamma$ with respect to the standard generating set of $W_\Gamma$.
		\item \label{thm:completions_torsion} The group $H$ has torsion if and only if there is some loop in the $1$-skeleton of $\Omega$ with label a reduced word whose letters are contained in a clique of $\Gamma$.
		\item \label{thm:completions_pi1} If $H$ is torsion-free, then $\Omega$ is non-positively curved and has fundamental group isomorphic to $H$.
		\item \label{thm:completions_loops} Any loop in $\Omega$ based at $B$ has label an element of $H$, and any reduced word which is an element of $H$ appears as the label of some loop based at $B$.
		\item \label{thm:completions_fi} Suppose additionally that $W_{\Gamma}$ does not split as a product with a finite factor 
		(equivalently, there does not exist a vertex of $V(\G)$ which is adjacent to every other vertex of $V(\G)$).
		Then $H$ is of finite index in $G$ if and only if a) $\Omega$ is finite, and b) for every vertex $v$ of $\Omega$ and for every $s \in V(\Gamma)$, there is an edge 
incident to $v$ in $\Omega$ which is labeled by $s$. 
	\end{enumerate}
\end{thm}

\subsection{Nonpositive sectional curvature}\label{sec:background_npsc}
We review the notion of nonpositive sectional curvature simplified  to our setting.
See \cite{Wise} and \cite{Wise-survey} for additional background.
	This material will only be needed in Section~\ref{sec:nonpos}.

Given a graph $\Theta$ with vertex set $V$ and edge set $E$, we define 
$\kappa(\Theta) := 2 - |V| +|E|/2$.
A \textit{spur} of $\Theta$ is an edge containing a vertex of valence one, and $\Theta$ is \textit{spurless} if it does not contain any spurs.

Let $\Sigma$ be a right-angled square complex.
Following~\cite{Wise}, we
say that $\Sigma$ has nonpositive sectional curvature if given any vertex $v$ of $\Sigma$ and any 
 subgraph $\Theta$ of the link of $v$ that is connected, spurless, and has at least one edge,
 it follows that $\kappa(\Theta) \le 0$. Similarly, we say that a simplicial graph $\Gamma$ has \textit{nonpositive sectional curvature} if all connected, spurless
 subgraphs $\Theta$ of $\Gamma$ 
 having at least one edge
 satisfy $\kappa(\Theta) \le 0$.

If a simplicial graph $\Gamma$ has nonpositive sectional curvature, then the (finite-index) commutator subgroup of the RACG $W_\Gamma$ is the fundamental group of a right-angled square complex of nonpositive sectional curvature. 

\section{$\Gamma$-partite graphs} \label{sec:partite_graphs}

Let $\G$ be a simplicial graph with vertex set $\{s_1, \dots, s_n\}$. We say that a simplicial graph~$\Delta$ is \emph{$\G$-partite} if its vertex set can be partitioned into $n$ non-empty, disjoint sets $A_1, \dots, A_n$ such that the following two conditions hold:
\begin{enumerate}
	\item No two vertices of $A_i$ are adjacent.
	\item For all $1 \le i < j \le n$, the subgraph $\Delta_{ij}$ of $\Delta$ induced by $A_i \cup A_j$ is connected if $s_i$ is adjacent to $s_j$ in $\Gamma$, and 
		is otherwise 
	 an edgeless graph, i.e., one with vertices but no edges. 
\end{enumerate}
We call $A_1, \dots, A_n$ the \textit{decomposition} of $\Delta$. 
We say that $\Delta$ has \textit{cycle connectors} if $\Delta_{ij}$ is either edgeless or a cycle for every $1 \le i < j \le n$. 
Similarly, we say that $\Delta$ has \textit{path connectors} if $\Delta_{ij}$ is either edgeless or a simple path for every $1 \le i < j \le n$.

The main result of this section is that given any simplicial graph $\Gamma$, there are $\Gamma$-partite graphs $\Delta$ that can be chosen so that $W_\Delta$ is hyperbolic and $2$-dimensional:

\begin{thm} \label{thm:square_free}
Given a finite simplicial graph $\G$, there exist
infinitely many
 $\G$-partite graphs which do not contain any simple $4$-cycles or $3$-cycles.  Additionally, these 
  $\G$-partite graphs can be constructed to have either cycle connectors or path connectors.
\end{thm}

\begin{figure}
	\includegraphics[scale=.7]{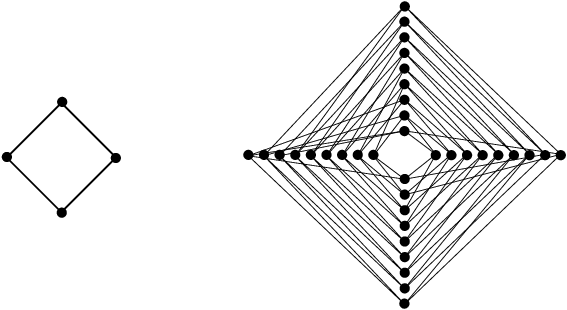}
	\caption{A graph $\Gamma$ on the left and, on the right, a $\Gamma$-partite graph with cycle connectors, no $3$-cycles and no $4$-cycles.} \label{fig:partite}
\end{figure}

\begin{proof}
	Let $\G$ have $E$ 	edges and let $V(\G)=\{s_1, \dots, s_n\}$. 
	 Choose a number $k> 8(3^E)$ 
	 which 
	 is not divisible by 3.  
	 We first construct a $\Gamma$-partite graph~$\Delta$ with decomposition $A_1, \dots, A_n$, such that $|A_i| =k$ for each $1 \le i \le n$, and such that 
$\Delta$ has cycle connectors. 
	
The vertex set of $\Delta$ is 
$A_1 \sqcup \cdots \sqcup A_n$.   
For each $1 \le i \le n$, we fix a labeling $a_0^i, \dots, a_{k-1}^i$ for the vertices of $A_i$.
To define the edge set of $\Delta$, we arbitrarily
	label the edges of $\Gamma$ as $e_1, \dots, e_E$, and arbitrarily orient each such edge.
	Fix such an edge $e_p$, and suppose $e_p$ is incident to the vertices $s_x$ and $s_y$ in $\Gamma$, and is oriented from $s_x$ to $s_y$.
Next, we define edges in $\Delta$ corresponding to $e_p$ as follows. 
Each vertex $a_l^x \in A_x$ is adjacent to exactly two vertices of $A_y$, namely $a_{l + 3^p}^y$ and $a_{l + 2(3^p)}^y$ (where arithmetic is modulo $k$). 
This choice of edges 
for each edge of $\G$
defines $\Delta$.
The following properties now follow:
	\begin{enumerate}
	 \item A vertex $a_{l'}^y \in A_y$ is adjacent to two vertices of $A_x$: $a_{l' - 3^p}^x$ and $a_{l' - 2(3^p)}^x$.
	 \item If $a_i^x$ and $a_j^x$ are distinct vertices of $A_x$ which share a neighbor in $A_y$, then 
	 $|i - j| \equiv 3^p \pmod k$.
	 Similarly,  if $a_{i'}^y$ and $a_{j'}^y$ are distinct vertices of $A_y$ which share a neighbor in $A_x$, then $|i' - j'| \equiv 3^p \pmod k$.
	 \item As $k$ is not divisible by $3$, 
	 it follows that $3^p$ and $k$ are relatively prime. 
	 	Thus by (2), $\Delta_{xy}$ is a simple $2k$-cycle.
	 \end{enumerate}
	
	It 
	follows that $\Delta$ is $\Gamma$-partite (by construction) and that $\Delta$ has cycle connectors (by 
	(3)). 
	
	We now check that there are no simple $4$-cycles in $\Delta$. First note that, for each $1 \le i \le n$, no simple $4$-cycle 
	has all its vertices
	in $A_i$, as there are no edges between vertices
	 in $A_i$. Moreover, no $4$-cycle 
	has all its vertices
in $A_i \cup A_j$ for some $1 \le i < j \le n$, as 
$\Delta_{ij}$ is either a simple $2k$-cycle with $k > 4$ or an edgeless graph. 
	
	Suppose there is a simple $4$-cycle $c$
	with vertices
	in $A_{i_0} \cup A_{i_1} \cup A_{i_2}$ for some $\{i_0, i_1, i_2\} \subset \{1, \dots, n\}$. Without loss of generality, we may assume that there 
are edges $e_p$, between $s_{i_0}$ and $s_{i_1}$, and 
$e_{q}$, between $s_{i_1}$ and $s_{i_2}$, in $\G$, and that
	$c$ contains two vertices $a_u^{i_1}$ and $a_v^{i_1}$ in $A_{i_1}$. As $c$ is a $4$-cycle, $a_u^{i_1}$ and $a_v^{i_1}$ have a common neighbor  in $A_{i_0}$ and a common neighbor in $A_{i_2}$. However, by 
		(2)
	above, we have that 
	 $|u - v| \equiv 3^p \pmod k$
	 and
	  $|u - v| \equiv 3^q \pmod k$.
	 This is a contradiction as $p \neq q$ and $k >  8(3^E) >  |3^p - 3^q|$. Thus, such a cycle $c$ cannot exist.
	
	Finally, suppose there is a 
	$4$-cycle $c$
	with vertices 
	in $A_{i_0} \cup A_{i_1} \cup A_{i_2} \cup A_{i_3}$ for some $\{i_0, i_1, i_2, i_3\} \subset \{1, \dots, n\}$, such that for each $0 \le j \le 3$, the cycle $c$  contains exactly one vertex in $A_{i_j}$ and there is an edge $e_{p_j}$ in $\Gamma$ between $s_{i_j}$ and $s_{i_{j+1}}$ (subscripts taken modulo $4$).
	
	To see that such a cycle $c$ cannot exist, we consider an arbitrary length $4$ path~$\gamma$ in $\Delta$ with startpoint 
		$a_l^{i_0} \in A_{i_0}$,
	 endpoint 
	$a_{l'}^{i_0} \in A_{i_0}$,
	and exactly one vertex in each of $A_{i_1}, A_{i_2}$ and $A_{i_3}$. 
	 It follows from the definition of edges of $\Delta$ and from 
	 	(1) above that:
	\[ l'  \equiv 	l  + \sum_{m = 0}^3 \epsilon_m 3^{p_m}    \pmod{k}\]
	where $\epsilon_m \in \{-1, -2, 1, 2\}$ for each $0 \le m \le 3$. 
	
	We claim that $\sum_{m = 0}^3 \epsilon_m 3^{p_m} \neq 0$. Suppose otherwise, for a contradiction. 
	Let $C$ (resp.~$C'$) be the sum of the positive  
	(resp.~absolute values of the negative)
	terms of $\sum_{m = 0}^3 \epsilon_m 3^{p_m}$.
	It follows that $C=C'$. However, as the numbers $p_0, p_1, p_2$ and $p_3$ are all distinct and positive, this implies that we have two ways to represent the number $C$ in base~$3$. This contradicts the basis representation theorem. We have thus shown that $\sum_{m = 0}^3 \epsilon_m 3^{p_m} \neq 0$.
	
	As $|\sum_{m = 0}^3 \epsilon_m 3^{p_m}| < k$ by our choice of $k$, and as $\sum_{m = 0}^3 \epsilon_m 3^{p_m} \neq 0$, it follows that  $l' \not\equiv l \pmod k$.
Thus, $a_l^{i_0}$ and $a_{l'}^{i_0}$ are distinct vertices.  We have then shown that there cannot be any 
$4$-cycle $c$ as described above. 

A similar
	argument
shows that $\Delta$ does not contain $3$-cycles. 
	This completes the construction of the claimed $\Gamma$-partite graph $\Delta$ with cycle connectors and no 
		simple
	$3$-cycles or $4$-cycles.
	To produce a $\Gamma$-partite graph $\Delta'$ with path connectors and no 
		simple
	$3$-cycles or $4$-cycles, we can, for each edge $(s_i, s_j)$ of $\Gamma$, remove an arbitrary edge from $\Delta_{ij}$.
\end{proof}

\begin{remark}
	The proof of the above theorem gives an explicit, easily implementable construction for $\Delta$, not just an existential one.
\end{remark}

Recall that a RACG is not one-ended if and only if its defining graph is either a clique, 
is disconnected, 
or is separated by a clique (see \cite[Corollary 2.4]{Dani}). From this and the structure of $\Gamma$-partite graphs, we immediately conclude:

\begin{lemma}\label{lem:one_ended}
	Let $\Gamma$ be a  finite simplicial graph, and let $\Delta$ be a $\Gamma$-partite graph. If $W_\Gamma$ is one-ended, then so is $W_\Delta$.
\end{lemma}

\section{Main Construction} \label{sec:main}

We give our main construction and result 
(Theorem~\ref{thm:functor} 
and Corollary~\ref{cor:nonquasiconvex})
in this section. 
Throughout this section, 
 we fix a simplicial graph $\Gamma$ with  vertex set $\{s_1, \dots, s_n\}$ and a $\Gamma$-partite graph $\Delta$ with decomposition $A_1, \dots, A_n$. 

Let $\Pi$ be a $\Delta$-labeled cube complex. Given $1 \le i \le n$ and vertices $u$ and $v$ of $\Pi$, a \textit{generalized edge between $u$ and $v$ with label $A_i$} is a collection of $|A_i|$ edges between $u$ and $v$ whose labels are in bijection with the vertices of $A_i$. Note that $u$ and $v$ could be equal.

Given a $\Gamma$-labeled cube complex $\Sigma$, we construct an associated $\Delta$-labeled cube complex $\overline \Sigma$. 
We call $\overline \Sigma$ the \textit{generalization of $\Sigma$ with respect to $\Delta$}, and we will always use the bar notation  to denote the generalization of a cube complex. When there is no confusion regarding the $\Gamma$-partite graph $\Delta$, we simply say that $\overline \Sigma$ is the generalization of $\Sigma$.

The vertex set of $\overline \Sigma$ is in bijective correspondence with that of $\Sigma$. The edges of $\overline \Sigma$ are defined as follows: for each edge $e$ of $\Sigma$ there is a  corresponding
generalized edge with label $A_i$ between the corresponding vertices of $\overline \Sigma$, where $s_i$ is the label of $e$. All edges of $\overline \Sigma$ are defined this way.

Note that the graph obtained from the $1$-skeleton of $\overline \Sigma$ by
collapsing generalized edges with label $A_i$ to single edges labeled by $s_i$, for each $1 \le i \le n$, is isomorphic to 
the 1-skeleton of $\Sigma$.  We call the resulting map $\overline \Sigma^1 \to \Sigma^1$ the \textit{collapsing} map.
A \textit{generalized $d$-cube} in $\overline \Sigma$ is a set of $2^d$ generalized edges whose image under the collapsing map is the $1$-skeleton of a $d$-cube of $\Sigma$.

Let $c$ be a $d$-cube in $\Sigma$, with $d \ge 2$, whose edges are labeled by distinct vertices $s_{i_1}, \dots, s_{i_d}$ of $\Gamma$.  By construction, there is a
	corresponding
 generalized $d$-cube $\overline c$ in  $\overline \Sigma^{1}$
 with generalized edges labeled by $A_{i_1}, \dots, A_{i_d}$,
 such that the image of $\bar c$ under the collapsing map is $c$. 
  For each $(a_1, \dots, a_d) \in A_{i_1} \times \dots \times A_{i_d}$ such that $a_1, \dots, a_d$ span a $d$-clique in $\Delta$, 
the complex $\overline \Sigma$ contains a single $d$-cube whose $1$-skeleton is the corresponding subset of $2^d$ edges of $\overline c$.
Note that even if some other $d$-cube in $\Sigma$ has the same boundary as $c$, there is only one $d$-cube in $\overline\Sigma$ for each such tuple. 
All $d$-cubes of $\overline \Sigma$, with $d \ge 2$, are defined this way. This completes the construction of~$\overline \Sigma$.
By construction, no cube identification operations can be performed on $\overline \Sigma$.

The next lemma is the key observation that allows us to prove our main theorem.

\begin{lemma} \label{lem:completion}
	Let $\Sigma$ be a 
finite
	$\Gamma$-labeled cube complex, and let $\overline \Sigma$ be the $\Delta$-labeled cube complex which is its generalization. Let $\Omega$ be a completion of $\Sigma$. Then there is a completion of $\overline \Sigma$ that is isomorphic to the generalization $\overline \Omega$ of $\Omega$.
\end{lemma}
\begin{proof}
	We first establish three claims, the first of which is obvious by construction. 
		\begin{claim}
	 If a cube identification operation can be performed on $\Sigma$ to obtain a new complex $\Theta$, then $\overline \Sigma$ is isomorphic to $\overline{\Theta}$.
	\end{claim}

	\begin{claim}
	 If a fold operation can be performed on $\Sigma$ to obtain a new complex $\Theta$, then a sequence of fold and cube identification operations can be performed on $\overline \Sigma$ to obtain $\overline{\Theta}$.
	\end{claim}

	\begin{proof}Suppose that there are edges $e$ and $f$ in $\Sigma$, each incident to a common vertex $v$ and each labeled by the same $s_i \in V(\Gamma)$. Then there are
			corresponding
		 generalized edges $\overline e$ and $\overline f$ in $\overline \Sigma$, each incident to a common vertex and each labeled by $A_i$. By performing $|A_i|$ folds, we identify $\overline e$ and $\overline f$ into a single generalized edge labeled by $A_i$. Next, we perform all possible cube identification operations to the resulting complex. This final resulting complex is then isomorphic to $\overline \Theta$.\end{proof}

	\begin{claim}
	If a cube attachment operation can be performed on $\Sigma$ to obtain a new complex $\Theta$, then a sequence of cube attachment, fold and cube identification operations can be performed on $\overline \Sigma$ to obtain $\overline \Theta$.
	\end{claim}
\begin{proof}Suppose first that such a cube attachment operation attaches a $d$-cube to $\Sigma$ with $d = 2$. It follows that there are edges $e$ and $f$ in $\Sigma$, incident to a common vertex and
 with labels $s_i$ and $s_j$ respectively, such that $s_i$ and $s_j$ are adjacent vertices of $\Gamma$. 
	In $\overline \Sigma$ we see corresponding generalized edges $\overline e$ and $\overline f$, 
	incident to
	a common vertex and labeled by $A_i$ and $A_j$ respectively.
	
	Let $x$ be an edge of $\Delta$ incident to 
		some 
	$u \in A_i$ and 
	$v \in A_j$. There is an edge in $\overline e$ with label $u$ and an edge in $\overline f$ with label $v$. We perform a square attachment operation to $\overline \Sigma$ by attaching a square with boundary label $uvuv$ to $\overline \Sigma$ along these two edges. We perform such a square attachment for every 
	edge $x$  between $A_i$ and $A_j$,
	and we denote the resulting complex by $\Pi$.
	
	Next, we perform all possible fold operations 
	between pairs of new edges that were attached in the process of attaching squares
	to $\overline \Sigma$.
	Let $\Pi'$ denote the final resulting complex from this sequence of operations. Note that $\overline \Sigma$ is naturally a subcomplex of both $\Pi$ and $\Pi'$. 
	
	We claim that $\Pi'$ is isomorphic to $\overline \Theta$. To prove this, it is enough to show that all vertices of $\Pi \setminus \overline \Sigma$ are identified to a common vertex of $\Pi'$. Thus, suppose that $z$ and $z'$ are vertices of $\Pi \setminus \overline \Sigma$ which are each contained in a square that was attached to $\overline \Sigma$. The labels of these squares are $abab$ and $a'b'a'b'$ for some $a, a' \in A_i$ and $b, b' \in A_j$. As $\Delta_{ij}$ is connected 
(by the definition of $\Gamma$-partite graphs), there is a path $a_1 = a, b_1 = b, a_2, b_2 \dots, a_{m-1}, b_{m-1}, a_m = a', b_m=b'$
	 in $\Delta$ with $a_l \in A_i$   and $b_l \in B_j$ for all $1 \le l \le m$. Correspondingly, there are squares with labels $a_lb_la_lb_l$ and $b_la_{l+1}b_la_{l+1}$ in $\Pi$ for all $1 \le l \le m-1$. After folding, all vertices contained in one of these squares which are not contained in $\overline \Sigma$ are identified. Thus, $\Pi'$ is isomorphic to $\overline \Theta$.
	
	  Suppose now that a cube attachment operation attaches a $d$-cube to $\Sigma$ with $d > 2$. From the previous case, it readily follows that there is a sequence of square attachment and fold operations that can be performed to  $\overline \Sigma$ to produce a complex whose 
	  $2$-skeleton
	  is isomorphic to that of $\overline \Theta$. It is clear that we can obtain $\overline \Theta$ by performing a series of cube attachments, folds,
	  and cube identifications to this complex.
	  \end{proof}
	We are now ready to prove the lemma. Let  $\Omega$ be a completion of $\Sigma$ obtained as the direct limit of a sequence $\Omega_0 = \Sigma \to \Omega_1 \to \Omega_2 \to \dots$ of complexes as in the definition of a completion.
	By the three claims above, there exists a sequence $\Pi_0 = \overline \Sigma \to \Pi_1 \to \Pi_2 \to \dots $ of complexes such that $\Pi_{i+1}$ is obtained from $\Pi_i$ by either a cube attachment, cube identification or fold operation, and such that there exists a subsequence $\Pi_0 = \Pi_{i_0}, \Pi_{i_1}, \Pi_{i_2}, \dots $ such that $\Pi_{i_j} = \overline{\Omega_j}$.
	Let $\Pi$ be the direct limit of $\Pi_0 = \overline \Sigma \to \Pi_1 \to \Pi_2 \to \dots $.
	It follows from the 
	definitions of direct limit and generalization of a complex
	that $\Pi$ is isomorphic to $\overline \Omega$. 
	
	In order to conclude that $\Pi=\overline \Omega$ is indeed a completion of $\overline \Sigma$, we must show that it is also folded and cube-full.	
		As $\Omega$ is a completion, it is folded and cube-full. 
	Let $\phi: \overline \Omega^1 \to \Omega^1$ be the collapsing map. 
	To see that $\overline \Omega$ is folded, let $e$ and $e'$ be edges of $\overline \Omega$ with the same label
and incident to a common vertex. Then, by the definition of a collapsing map, $\phi(e)$ and $\phi(e')$ are incident to a common vertex and have the same label.  
As $\Omega$ is folded, $\phi(e) = \phi(e')$.
Thus, $e$ and $e'$ must lie in a common generalized edge. As they have the same label, it follows that $e=e'$.
Finally, 
	since no cube identification operations can be performed on $\overline \Omega$ by definition, we conclude that $\overline \Omega$ is folded.
	
	To see that $\overline \Omega$ is cube-full, consider a set $e_1, \dots, e_d$ of edges in $\overline \Omega$  with $d \ge 2$, all incident to a common vertex and with labels $t_1, \dots, t_d$ corresponding to a $d$-clique of $\Delta$. For $1 \le i \le d$, let $A_i$ be the set in the decomposition of $\Delta$ that contains the vertex $t_i$.
	As there are no edges between vertices of $A_i$ for any  $i$, it follows that $A_i \neq A_j$ for all $1 \le i 	\neq j	\le d$.
	Thus, by the structure of $\Gamma$-partite graphs, the labels of $\phi(e_1), \dots, \phi(e_d)$ span a $d$-clique in $\Gamma$. As $\Omega$ is cube-full, $\phi(e_1), \dots, \phi(e_d)$ are contained in a common $d$-cube. Consequently, as $\overline \Omega$ is the generalization of $\Omega$, there is a $d$-cube in $\overline \Omega$ containing $e_1, \dots, e_d$.
	Thus, $\overline \Omega$ is cube-full. 
\end{proof}

Let $H$ be a 
 subgroup of $W_\Gamma$ generated by a finite set $T$ of words in $W_\Gamma$, and let $R$ be the rose graph associated to $T$. 
Let $\overline R$ be the generalization of $R$ with respect to $\Delta$.
Let $l_1, \dots, l_m$ be a set of loops in $\overline R$, based at the base vertex, which generate $\pi_1(\overline R)$, and let $\overline T = \{w_1, \dots, w_m\}$ be the set of labels of these loops.
We define the \textit{generalization $\overline H$ of $H$ with respect to $\Delta$ and $T$} to be the finitely generated subgroup of $W_\Delta$ that is generated by $\overline T$. 
We remark that the isomorphism class of $\overline H$ very much depends on $T$ and $\Delta$, but not on $\overline T$.  
When there is no confusion regarding $T$ and $\Delta$, we simply say that $\overline H$ is the generalization of $H$. We will always use the bar notation as such to denote the generalization of a group.

\begin{thm} \label{thm:functor}
	Let $\Gamma$ be a simplicial graph, and let $\Delta$ be a $\Gamma$-partite graph. 
Let $H$ be a subgroup of $W_\Gamma$ generated by a finite set $T$ of words in $W_\Gamma$, and let $\overline H < W_\Delta$ be its generalization with respect to $\Delta$ and $T$. Then
\begin{enumerate}
	\item 
$H$ is quasiconvex in $W_\Gamma$ if and only if $\overline H$ is quasiconvex in $W_\Delta$, where quasiconvexity is with respect to the standard generating sets of $W_\Gamma$ and  $W_\Delta$ respectively. 
	\item If $H$ is torsion-free, then so is $\overline H$.
	\item If $H$ is not normal in $W_\Gamma$, then $\overline H$ is not normal in $W_\Delta$.
	\item Suppose additionally that $W_{\Gamma}$ does not split as a product with a finite factor 
		(equivalently, there does not exist a vertex of $V(\G)$ which is adjacent to every other vertex of $V(\G)$).
	Then $H$ is a finite index subgroup of $W_\Gamma$ if and only if $\overline H$ is a finite index subgroup of $W_\Delta$.
\end{enumerate}
Additionally, if $\Delta$ does not contain induced $4$-cycles, 
 (resp.~$3$-cycles), then $W_\Delta$ is hyperbolic (resp.~$2$-dimensional).
\end{thm}

Recall that an induced cycle is a cycle which is equal to the subgraph induced by its vertices.  

\begin{proof}
	The claims regarding $W_\Delta$ being hyperbolic or $2$-dimensional follow from a theorem of Moussong \cite{Moussong} and by definition respectively.
	
	Let $R$ be the rose graph associated to $T$. Let $\overline R$ be the generalization of $R$, and let $\overline T$ be the set of labels of a finite set of generators of $\pi_1(\overline R, B)$, where $B$ is the base vertex. By definition, the generalization $\overline H$ is generated by $\overline T$. Let $R'$ be the rose graph associated to $\overline T$. It readily follows that $\overline R$ can be obtained from $R'$ by a sequence of fold operations. 
	In particular, a completion of $\overline R$ is a completion of $\overline H$. 
	
	Let $\Omega$ be a completion of $R$. By Lemma \ref{lem:completion}, $\overline \Omega$ is a completion of $\overline R$ and  is therefore a completion of $\overline H$. 
	As $\Omega$ is infinite if and only if $\overline \Omega$ is infinite, (1) immediately follows from Theorem \ref{thm:completions}(\ref{thm:completions_qc}). 
	
	By Theorem \ref{thm:completions}(\ref{thm:completions_torsion}), if $\overline H$ has torsion, then there is a loop in $\overline \Omega$ whose label is a 
	reduced word $w$ with letters in a clique of $\Delta$.  It follows from the definition of a $\G$-partite graph that $w$ uses at most one letter from each set in the decomposition of $\Delta$.  Applying the collapsing map, one obtains a loop in $\Omega$ with label a word $w'$ using letters in a clique of $\G$, with each letter being used at most once.  It follows that $w'$ is reduced. 
	Again by Theorem \ref{thm:completions}(\ref{thm:completions_torsion}), this implies that $H$ has a torsion element. Thus, (2) follows. 
	
	We define a map $\phi: V(\Delta) \to V(\Gamma)$ by, for each $1 \le i \le n$ and each $a \in A_i$, setting $\phi(a) = s_i$.
	Moreover, if two generators $s, t \in V(\Delta)$ of $W_\Delta$ commute, then it follows from the definition of a $\Gamma$-partite graph that $\phi(s)$ commutes with $\phi(t)$ in $W_\Gamma$. Thus, $\phi$ defines a homomorphism $\phi: W_\Delta \to W_\Gamma$ which is readily seen to be surjective.
	
	To see (3), suppose that $H$ is not normal in $W_\Gamma$. Let $h$ be a word representing an element of $H$ and $g$ be a word representing an element of $W_\Gamma$ so that 
	$ghg^{-1}$
	does not represent an element of $H$. 
	By Theorem \ref{thm:completions}(\ref{thm:completions_loops}) and as $\overline \Omega$ is a generalization of $\Omega$, there exists a word $\overline h$, representing an element of $\overline H$, so that $\phi(\overline h) = h$. As $\phi$ is surjective, there is a word $\overline g$ representing an element of $W_\Delta$ so that $\phi(\overline g) = g$. 
	Let $\overline w$ be a reduced word equal to $\overline g \overline h \overline g^{-1}$ in $W_\Delta$. Suppose, for a contradiction, that $\overline w$ represents an element of $\overline H$. By Theorem \ref{thm:completions}(\ref{thm:completions_loops}), we have that there is a loop in $\overline \Omega$ with label $\overline w$ which is based at the base vertex. By Lemma \ref{lem:completion}, there is a loop with label $\phi(\overline w)$ in $\Omega$ based at the base vertex. However, again by Theorem~\ref{thm:completions}(\ref{thm:completions_loops}), this implies that $\phi(\overline w) \in H$. This is a contradiction as $\phi(\overline w)$ is equal to $ghg^{-1}$ in $W_\Gamma$. Thus, $\overline g \overline h \overline g^{-1}$ does not represent an element of $\overline H$, and so $\overline H$ is not normal in $W_\Delta$.
	 
	Claim (4) follows from Theorem~\ref{thm:completions}(\ref{thm:completions_fi}).
\end{proof}

The next corollary follows directly from Theorem \ref{thm:functor} and Theorem \ref{thm:square_free}.

\begin{cor} \label{cor:nonquasiconvex}
	Let $H$ be a finitely generated
	non-quasiconvex subgroup of a (possibly non-hyperbolic) RACG $W_\Gamma$.
	 Then there exists an infinite 
	set $\{\Delta_i \mid i \in \mathbb N\}$
	 of $\Gamma$-partite graphs such that, for each $i$, $W_{\Delta_i}$ is $2$-dimensional, hyperbolic and contains the 
	 generalization $\overline H$ of $H$ (with respect to $\Delta_i$ 
	 and any finite generating set for $H$) 
	 as a non-quasiconvex subgroup.
\end{cor}

We finish this section by showing how Theorem \ref{thm:functor} gives easy examples of non-quasiconvex subgroups of hyperbolic groups. 

\begin{example} \label{ex:bipartite}
Let $\Gamma$ be the graph with vertex set $\{s_1, \dots, s_{2k} \}$, where $k \ge 2$, and with edges between $s_i$ and $s_j$ for all odd $i$ and even $j$, i.e., $\G$ is a complete bipartite graph.
	Let $w$ be the word $s_1s_2 \dots s_{2k}$, and let $H$ be the cyclic subgroup of $W_\Gamma$ generated by $w$. The subgroup $H$ can be seen to not be quasiconvex in $W_\Gamma$ by noting that $W_\Gamma = W_{ \{s_1, s_3, \dots, s_{2k-1} \}} \times W_{ \{s_2, s_4, \dots, s_{2k} \}}$ and that $w^n$ is equal, in $W_\Gamma$, to the word $(s_1s_3 \dots s_{2k-1})^n(s_2s_4 \dots s_{2k})^n$. Alternatively, one can
	conclude non-quasiconvexity
	 from Theorem~\ref{thm:completions}(\ref{thm:completions_qc}) by noting that $H$ has a completion which is an infinite cylinder that is tiled by squares 
	(see for instance Figures 2 and 6 of \cite{DL}). 
	
	By Theorem \ref{thm:square_free}, there exists a $\Gamma$-partite graph $\Delta$ such that $W_\Delta$ is $2$-dimensional and hyperbolic. Then, by Theorem \ref{thm:functor}, the generalization $\overline H$  with respect to $\Delta$ and $\{w\}$
	is a non-quasiconvex subgroup of $W_\Delta$. Additionally, as $H$ is not normal in $W_\Gamma$, 
	Theorem~\ref{thm:functor} implies that $\overline H$ is not normal in $W_\Delta$.
\end{example}

The following example gives non-quasiconvex subgroups with generators as short as possible, as discussed in the introduction.
\begin{example} \label{ex:2_gen}
	Let $\Gamma$ be a $4$-cycle with cyclically ordered vertices $\{s_1, \dots, s_4 \}$. Let $w = s_1 s_2$, $w' = s_3s_4$ and $T = \{w, w'\}$.  Note that $w$ and $w'$ are each of order  two in $W_\Gamma$. 
	The infinite dihedral subgroup $H < W_\Gamma$ generated by $T$ is 
	easily
	 seen to be non-quasiconvex in $W_\Gamma$. The rose graph associated to $T$ and the $1$-skeleton of a completion for it is shown in Figure \ref{fig:2_gen}.
	
	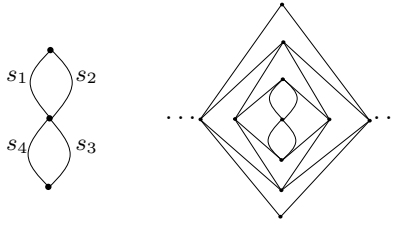
\begin{figure} 
		\begin{overpic}[scale=.5]{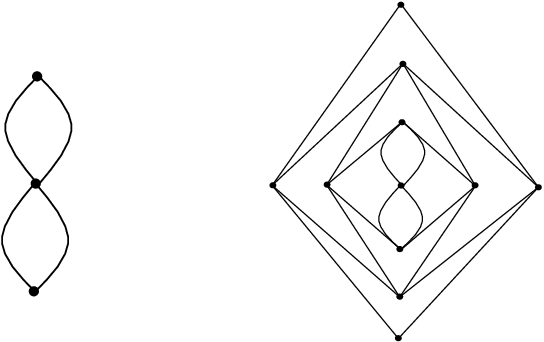}
		\put(-6,40){\small $s_1$}
		\put(14,40){\small $s_2$}
		\put(-6,20){\small $s_4$}
		\put(14,20){\small $s_3$}
		\put(40,29){\dots}
		\put(100,29){\dots}
		\end{overpic}
		\caption{On the left is the rose graph $R$ of Example \ref{ex:2_gen}, and on the right is the infinite $1$-skeleton of a completion $\Omega$ of $R$. The graph $\overline R$ 
		 (resp.~$\overline \Omega^1$)
		is obtained by replacing edges of $R$ 
		 (resp. $\Omega^1$) by generalized edges.}
		\label{fig:2_gen}
	\end{figure}
	
	By Theorem \ref{thm:square_free}, there exists a $\Gamma$-partite graph $\Delta$ with decomposition $A_1, \dots, A_4$ such that $W_\Delta$ is hyperbolic and $2$-dimensional (such as the one in Figure \ref{fig:partite}). Thus, by Theorem \ref{thm:functor}, the generalization $\overline H$ of $H$ is a non-quasiconvex subgroup of $W_\Delta$. Moreover, by definition $\overline H$ is generated by the labels of all simple loops in $\overline R$ based at the base vertex, where $R$ is the rose graph associated to $T$. Thus, $\overline H$ is generated by all words in $V(\Delta)$ of the from $aa'$ with either $\{a, a'\} \subset A_1 \cup A_2$ or $\{a, a'\} \subset A_3 \cup A_4$. In particular, all generators of $\overline H$ have length $2$.
\end{example}

Examples~\ref{ex:zn} and~\ref{ex:not_product} below provide additional examples of non-quasiconvex subgroups of RACGs, which can be used to produce non-quasiconvex subgroups of hyperbolic groups.

\begin{example} 
(Products)
\label{ex:zn}
	Given a join graph $\Gamma = \Gamma_1 \star \Gamma_2$, where
	$\Gamma_1$ and $\Gamma_2$ are not cliques, the RACG 
$W_\Gamma = W_{\Gamma_1} \times W_{\Gamma_2}$ 
	contains many non-quasiconvex subgroups. 
	Similarly, let $\Gamma_0$ be  a $4$-cycle, and for integers $i > 0$, let $\Gamma_i$ be the suspension of $\Gamma_{i-1}$. For $i \ge 0$, $W_{\Gamma_i}$ is virtually isomorphic to $\mathbb{Z}^{i+2}$ and also contains many non-quasiconvex subgroups.
\end{example}

	In the following example, $W_\Gamma$ is not a product. More complicated examples that are not products can be similarly constructed.
	\begin{example}\label{ex:not_product}
		Let $\Gamma$ be a graph consisting of a $6$-cycle with cyclically ordered vertices $s_1, \dots, s_6$ and an edge between $s_1$ and $s_4$. Let $T = \{w = s_1s_2s_3s_4, w'= s_1s_4s_5s_6 \}$, and let $H < W_\Gamma$ be generated by $T$. Let $S_1$ and $S_2$ be the $4$-cycles of $\Gamma$ induced by $ \{s_1, s_2, s_3, s_4\}$ and $\{s_1, s_4, s_5, s_6\}$ respectively. The group $W_\Gamma$ is the amalgamation of $W_{S_1}$ and $W_{S_2}$ over a finite group. As $\langle w \rangle$ is not quasiconvex in $W_{S_1}$, it readily follows that $H$ not quasiconvex in $W_\Gamma$.
		 This can also 
		  be seen using completions.
	\end{example}

\section{Non-quasiconvex free subgroups} \label{sec:nonpos}
In this section, we 
construct non-quasiconvex free subgroups 
with the additional property that they are not normal in any finite-index subgroup of the parent group 
(see Theorem~\ref{thm:free_subgroup}).  We also construct
non-locally quasiconvex hyperbolic groups with nonpositive sectional curvature (see Theorem~\ref{thm:non_pos}).

	The RACG $W_\Delta$ in the following
	theorem 
	can be chosen to be hyperbolic and $2$-dimensional by Theorem~\ref{thm:square_free}.
\begin{thm} \label{thm:free_subgroup}
	 As in Example \ref{ex:bipartite}, let 
$\Gamma =  \{s_1, s_3, \dots, s_{2k-1} \} \star \{s_2, s_4, \dots, s_{2k}\}$ 
	 be the complete bipartite graph with $k \ge 2$ and let $H < W_\Gamma$ be the cyclic subgroup generated by $s_1s_2\dots s_{2k}$. Let $\Delta$ be a $\Gamma$-partite graph with path connectors and with decomposition $A_1, \dots, A_{2k}$ satisfying $|A_i| \ge 2$ for each $i$.
	Then the generalization $\overline H$ is a finitely generated, non-quasiconvex, non-normal, 
	free subgroup of the one-ended group $W_\Delta$. Moreover, when $k > 2$, then $\overline H \cap K$ is not normal in $K$ for any finite index subgroup $K < W_\Delta$.
\end{thm}
\begin{proof}
Note that $W_\G$ is one-ended by Lemma \ref{lem:one_ended}.
Let $R$ denote a $2k$-cycle with label $s_1s_2\dots s_{2k}$.  Recall that by definition, $\overline H$ is the subgroup of $W_\Delta$ generated by the labels of a set of generators of $\pi_1(\overline R, B)$, where $B$ is the base vertex. Note also that $\pi_1(\overline R, B)$ is free of some rank $l$.

There exists a  completion $\Omega$ of $R$ which consists of a tiling of $\mathbb R \times S^1$ by squares 
(see Figures~2 and~6 of~\cite{DL}).
 By Lemma~\ref{lem:completion}, the generalization $\overline \Omega$ of $\Omega$ is a completion for $\overline H$.  It readily follows from the description of $\Omega$ and Theorems~\ref{thm:functor} and~\ref{thm:completions} that $\overline H$ is non-quasiconvex, non-normal and torsion-free.  
Theorem~\ref{thm:completions}(\ref{thm:completions_pi1}) then implies that $\overline H \cong \pi_1(\overline \Omega) $.  Thus, we show that $\overline H$ is free by showing that 
$\pi_1(\overline \Omega)$ is a free group of rank $l$.

We first introduce some new complexes.  
For $1 \le i \le 2k$, let $R_i$ denote the subgraph of $\overline R$ consisting of the generalized edges with labels $A_i$ and $A_{i+1}$ (taken modulo $2k$), and let $v_i$ denote the vertex incident to both these generalized edges.  Choose one edge in each of these two generalized edges to form a maximal tree, leading to a choice of basis loops for $\pi_1(R_i)$.  Let $H_i$ be the subgroup of $W_\Delta$ generated by the labels of these loops,
which are all infinite order elements of length two.
By Theorem 3.17 from~\cite{DL-visual_raags}, $H_i$ is a free group with this generating set as a free basis.
Let $\Omega_i$ denote the generalization of a single $2$-cube with boundary label $s_i s_{i+1} s_i s_{i+1}$.  It can be seen that $\Omega_i$ is a completion for $H_i$. 
Moreover, since $H_i$ is free, it is torsion-free, and so $\pi_1(\Omega_i) \cong H_i  \cong \pi_1(R_i)$.

It is evident from the description of $\Omega$ and the definition of a generalization of a complex, that $\overline \Omega$ is a direct limit of a sequence of complexes $\Sigma_i$, where $\Sigma_0 = \overline R$ and $\Sigma_{j+1}$ is obtained from $\Sigma_j$ by gluing a copy of $ \Omega_i$ for some $i$ to $\Sigma_{j}$ along a copy of $ R_i$ in each of these.  

Now for $j\ge0$, van Kampen's theorem says that $\pi_1(\Sigma_{j+1})$ is the pushout of the diagram 
$\pi_1(\Sigma_j) \leftarrow \pi_1(R_i) \rightarrow \pi_1(\Omega_i)$.   
From what we proved above, the 
map $\pi_1(R_i) \rightarrow \pi_1(\Omega_i)$ is an isomorphism, and so its inverse provides a natural map 
$\pi_1(\Omega_i) \to \pi_1(\Sigma_j)$.  By the universal property of pushouts, this map, together with the identity map of $\pi_1(\Sigma_j)$ induces a unique map $\phi:\pi_1(\Sigma_{j+1}) \to \pi_1(\Sigma_{j})$.  Moreover, the natural map $\pi_1(\Sigma_{j}) \to \pi_1(\Sigma_{j+1})$ is an inverse for $\phi$.  Thus, $\pi_1(\Sigma_{j}) \cong \pi_1(\Sigma_{j+1})$.  This, together with the fact that $\pi_1(\Sigma_0)$ is a free group of rank $l$ completes the proof that $\overline H$ is free of rank $l$. 

To prove the claim regarding normality,
choose distinct $a, a' \in A_1$ and $b, b' \in A_3$. 
Recall that there are no edges between two vertices of $A_i$, and as $s_1$ is not adjacent to $s_3$ in $\Gamma$,
there are no edges between $\{a, a'\}$ and $\{b, b'\}$.
It follows by Tits' solution to the word problem (see \cite[Theorem 3.4.2]{Davis}) that for every $p, q \in \mathbb{Z} \setminus \{0\}$, 
the words
$(aa')^p$ and $(bb')^q(aa')^p(b'b)^q$ are reduced. 
In $\Omega$, the base vertex is contained in four distinct edges, labeled by $s_1$, $s_2$, $s_{2k}$ and $s_{2k-1}$. Thus, by Theorem~\ref{thm:completions}(\ref{thm:completions_loops}) and the structure of $\overline \Omega$, we have that, for every $p, q \in \mathbb{Z} \setminus \{0\}$, 
the word
$(aa')^p$ represents an element of $\overline H$ and $(bb')^q(aa')^p(b'b)^q$ does not represent an element of $\overline H$. Now, given any finite index subgroup $K < W_\Delta$, there exists some $p, q \in \mathbb{Z} \setminus \{0\}$ such that $(bb')^q \in K$ and $(aa')^p \in K \cap \overline H$. The claim follows.
\end{proof}

The next lemma constructs $\Gamma$-partite graphs with nonpositive sectional curvature
 (see Section~\ref{sec:background_npsc} for the definition).
These graphs are used in Theorem~\ref{thm:non_pos} below. 

\begin{lemma} \label{lem:sec_curv}
	Let $\Gamma$ be a $4$-cycle, and let $\Delta$ be a $\Gamma$-partite graph with path connectors. Then $\Delta$ has nonpositive sectional curvature.
\end{lemma}
\begin{proof}
	Let $s_1, \dots, s_4$ be the labels of a cyclic ordering of the vertices of $\Gamma$, and let $A_1, \dots, A_4$ be the corresponding decomposition of $\Delta$.
	Let $\Theta$ be a connected, spurless subgraph of $\Delta$ which contains at least one edge. Let $V_i$ be the number of vertices of $\Theta$ contained in $A_i$, and let $E_i$ be the number of edges of $\Theta$ connecting a vertex of $A_i$ to a vertex of $A_{i+1}$ (taken modulo 4). We need to check that 
	$\kappa(\Theta)= 2 - \sum_{i=1}^4 V_i + \frac{1}{2}\sum_{i=1}^{4}E_i$ is at most $0$.
	
	As the edges in $\Delta$ between $A_i$ and $A_{i+1}$ form a 
	simple 
	path, we have that $E_i \le 2 V_i - 1$ whenever $V_i =  V_{i+1} \neq 0$ and  $E_i \le 2 \min(V_i, V_{i+1})$ whenever $V_i \neq V_{i+1}$. 
	
	Suppose first that $V_i \neq 0$ for all $i$. We then get:
	\[\kappa(\Theta) \le 2 - \sum_{i=1}^4 V_i + \sum_{i=1}^4 \epsilon_i (V_i - \frac{1}{2}) + \sum_{i=1}^4(1 - \epsilon_i)  \min(V_i, V_{i+1}) \]
	where $\epsilon_i = 1$ if  
	$V_i = V_{i+1}$ 
	and $\epsilon_i = 0$ otherwise. Regrouping terms and noting that $\epsilon_i (V_i - \min(V_i, V_{i+1})) = 0$ for all $i$, we get: 
	\begin{align*}
		\kappa(\Theta) &\le 2 - \sum_{i=1}^4 V_i - \frac{1}{2} \sum_{i=1}^4 \epsilon_i + \sum_{i=1}^4 \min(V_i, V_{i+1}) + \sum_{i=1}^4 \epsilon_i (V_i - \min(V_i, V_{i+1}))  \\
		&= 2  -A-B, 
	\end{align*}
	where
	$$A := \sum_{i=1}^4 V_i -\sum_{i=1}^4 \min(V_i, V_{i+1}) \text{ and } B := \sum_{i=1}^4\frac{\epsilon_i}{2}.$$  
		Note that $A \ge 0$, as $V_i - \min(V_i, V_{i \pm 1}) \ge 0$ for any $i$.
	We prove below that $A+B \ge 2$.  From this, it follows that $\kappa(\Theta) \le 2-A - B \le 0$.
	
	Let $M$ and $m$ respectively be the maximum and minimum of $\{V_i \,| \,1 \le i \le 4\}$.  If $M=m$, then we necessarily have that $\epsilon_i = 1$ for all $i$.  It follows that 
	$B=2$.  Thus, $A+B \ge 2$, since $A\ge 0$.
	Now suppose $M > m$.  
	Without loss of generality, assume that $V_1=m$. Then 
	$$A = m + V_2+V_3+V_4  - (m +\min(V_2, V_{3})  + \min (V_3, V_{4}) +m).$$
	We have $V_i =M$ for some $i \in \{2, 3, 4\}$.  In each of these cases, one readily verifies that $A \ge M-m$, (using the observation that $V_i - \min(V_i, V_{i \pm 1}) \ge 0$ for any $i$).   From this, it follows that if $M-m\ge 2$, then $A+B \ge M-m +0 \ge 2$.   
	Finally, suppose $M-m =1$, so that $A \ge 1$.  
	Observe that if $\epsilon_i = 1$ for some $i$, then there exists $j \neq i$ such that $\epsilon_j = 1$.  In this case, $B \ge 1$, and once again $A + B \ge 2$.  In the last remaining case, namely when $M-m=1$ and $\epsilon_i=0$ for all $i$, we necessarily have $V_2=V_4=m+1$ and $V_3=m$, and it is easily verified that $A \ge 2$.  Therefore,  $A+B \ge 2$.
	We have thus verified that the sectional curvature is at most $0$ when $V_i \neq 0$ for all $i$.
	
	If  $V_i = V_j = 0$ for some $i \neq j$, then $\Theta$ necessarily contains a spur. This follows as $\Theta$ is connected, contains at least one edge and the edges in $\Delta$ between $A_i$ and $A_{i+1}$ form a simple path. Thus, it cannot be that $V_i = V_j = 0$ for some $i \neq j$.
	
	The only remaining case to check, up to relabeling, is 
		if
	 $V_4 = 0$ and $V_1, V_2$ and $V_3$ are non-zero.
	 Suppose this is the case. Each vertex of $A_1$ (respectively of $A_3$), is adjacent to at most two vertices of $A_2$. As each edge of $\Theta$ is incident to a vertex of $A_1 \cup A_3$, we see that the total number of edges of $\Theta$ must be at most $2(V_1 + V_3)$. We thus get:
	\[ \kappa(\Theta) \le 2 - \sum_{i=1}^3V_i + (V_1 + V_3)\]
	If $V_2 > 1$, this implies that $\kappa(\Theta) \le 0$ and we are done. On the other hand, it cannot be that $V_2 = 1$, for then 
every 
	 vertex of $A_1 \cap \Theta$
	  must be a spur. This establishes the claim.
\end{proof}

By applying 
Theorem~\ref{thm:free_subgroup}
 and 
Lemma~\ref{lem:sec_curv}, we now
obtain explicit examples of non-locally quasiconvex, hyperbolic groups 
with nonpositive sectional curvature.

\begin{thm} \label{thm:non_pos}
	Let $\Gamma$ be a $4$-cycle, and let $\Delta$ be a $\Gamma$-partite graph with path connectors, no $3$-cycles and no $4$-cycles. Then $W_\Delta$ has a finite-index, torsion-free, hyperbolic subgroup $K$ that is the fundamental group of a right-angled square complex with nonpositive sectional curvature. Moreover, $K$ contains a finitely generated, non-quasiconvex, free subgroup.	In particular, $K$ is not locally quasiconvex.
\end{thm}
\begin{proof}
	Let $K$ be the commutator subgroup of $W_\Delta$. Then $K$ is hyperbolic, as it is a finite-index subgroup of the hyperbolic group $W_\Delta$. Morever, $K$ is the fundamental group of a right-angled square complex with nonpositive sectional curvature, as $\Delta$ has nonpositive sectional curvature by Lemma \ref{lem:sec_curv}.
	Let $\overline H$ be the finitely generated, non-quasiconvex free subgroup of $W_\Delta$ given by Theorem \ref{thm:free_subgroup}. 
	Then
	$\overline H \cap K$ is a finitely generated, non-quasiconvex, free subgroup of~$K$.
\end{proof}

\section{Non-finitely presentable, non-quasiconvex subgroups} \label{sec:not_free}

In this section we show how $\Gamma$-partite graphs with cycle connectors provide examples of finitely generated, non-quasiconvex subgroups of hyperbolic RACGs which contain closed surface subgroups (Theorem \ref{thm:surface_subgroups}). We also give, using our method, a construction of finitely generated, non-finitely presentable subgroups of hyperbolic RACGs (Theorem \ref{thm:not_fp}).

\begin{lemma}\label{lem:surface_subgroups}
	Let $C$ be a cycle with cyclically-ordered vertices $p_1, p_2, \dots, p_{2k}$, where
	 $k \ge 3$. 
	Let 
	\[A = \{p_1p_i ~|~ i \text{ is odd}, i > 1\}  \text{ and } B = \{p_2p_i ~|~ i \text{ is even}, i > 2\} \]
	Then the subgroup $K < W_C$ generated by $A \cup B$
	is isomorphic to the fundamental group of a closed hyperbolic surface.
\end{lemma}
\begin{proof}
The group $W_C$ is well-known to act geometrically by reflections on the hyperbolic plane (see Theorem 6.4.3 and Example 6.5.3 of~\cite{Davis} for instance). 
Let $\Gamma$ be a graph with exactly two vertices, $t_1$ and $t_2$, and with one edge between them.
Note that $W_\Gamma$ is a finite group.
Also, note that $C$ is a $\Gamma$-partite graph with decomposition $A_1 = \{s_1, s_3, \dots, s_{2k-1}\}$ and $A_2 = \{s_2, s_4, \dots, s_{2k}\}$. Let $H$ be the (trivial) subgroup of $W_\Gamma$ generated by the set of words $T = \{t_1t_1, t_2t_2\}$. It follows that $K$ is the generalization of $H$ with respect to $C$ and $T$. As $H$ is a torsion-free, finite-index subgroup of $W_\Gamma$, it follows by Theorem~\ref{thm:functor} that $K$ is a torsion-free and finite-index subgroup of $W_C$.
 	Thus, $K$ is isomorphic to the fundamental group of a closed hyperbolic surface.
\end{proof}

\begin{thm} \label{thm:surface_subgroups}
	Let $\Gamma$ be a simplicial graph, $\Delta$ be a $\Gamma$-partite graph with cycle connectors, and 
	 $H$ be a non-quasiconvex subgroup of $W_\Gamma$. Then $\overline H$ is not quasiconvex in $W_\Delta$ and contains a closed surface subgroup.	
\end{thm}
\begin{proof}
	Let $\Omega$ be a completion of $H$. The generalization $\overline \Omega$ is a completion of $\overline H$ by Lemma~\ref{lem:completion}.
	As $H$ is not quasiconvex in $W_\Gamma$, Theorem~\ref{thm:completions} implies that 
	$\Omega$ is infinite.
	In particular, some cube attachment operation must have been performed to construct $\Omega$, and it follows that $\Omega$ contains a square $c$
	with label $s_is_js_is_j$, where $s_i$ and $s_j$ are adjacent vertices of $\Gamma$.
	We see a corresponding generalized square $\overline c$ (whose image under the collapsing map  is $c$) in $\overline \Omega$ with label $A_iA_jA_iA_j$. 
	
	Let $v$ be a vertex of $\overline \Omega$ that is incident to both a generalized edge $\overline e$ labeled by $A_i$ and a generalized edge $\overline f$ labeled by $A_j$.
	Note that $\Delta_{ij}$
	is a cycle by hypothesis, and the subgroup generated by $A_i \sqcup A_j$ (the set of vertices of $\Delta_{ij}$) is isomorphic to 
	$W_{\Delta_{ij}}$ (see \cite[Theorem 4.1.6]{Davis}).
	Let $K$ be the subgroup 
	of $W_{\Delta_{ij}}$ (hence of $W_\Delta$)
	generated by the labels of all simple loops based at $v$ and contained in $\overline e \cup \overline f$.
	Then $K$ is a closed hyperbolic surface subgroup by Lemma~\ref{lem:surface_subgroups}. 
	Let $w$ be the label of a path in $\overline \Omega$ from the base vertex to 
	$v$. 
	By Theorem~\ref{thm:completions}(\ref{thm:completions_loops}), $\overline H$ contains  $wKw^{-1}$.
\end{proof}

\begin{thm} \label{thm:not_fp}
	Let $\Gamma$ be a $4$-cycle, and let $\Delta$ be a $\Gamma$-partite graph with cycle connectors, no $3$-cycles and no $4$-cycles. Then $W_\Delta$ is $2$-dimensional, hyperbolic and contains a finitely generated, non-finitely presentable (and in particular, non-quasiconvex) subgroup. 	
\end{thm}
\begin{proof}
	We immediately have that $W_\Delta$ is $2$-dimensional and hyperbolic as $\Delta$ has no $3$-cycles and no $4$-cycles.
	Let the vertices of $\Gamma$ be cyclically labeled $s_1, \dots, s_4$. Let $H$ be the subgroup of $W_\Gamma$ generated by the word $s_1s_2s_3s_4$. 

	Let $\overline H$ be the generalization of $H$. Let $K$ be the commutator subgroup of $W_\Delta$, and let $\overline H' := \overline H \cap K$. As $W_\Delta$ is $2$-dimensional, $K$ has cohomological dimension at most $2$. 
	We will show that $\overline H'$ is a normal, infinite-index, infinite, non-free subgroup of $K$. 
	It will then follow from a well-known theorem of Bieri \cite[Theorem~B]{Bieri-normal-subgroups} 
	 that $\overline H'$ is not finitely presentable.
	
	As $H$ is torsion-free and has infinite index in $W_\Gamma$,  we have by Theorem \ref{thm:functor} that $\overline H$ is torsion-free and has infinite index in $W_\Delta$. In particular, $\overline H'$ is infinite and has infinite index in $K$. 
	By Theorem~\ref{thm:surface_subgroups}, $\overline H$ contains a closed surface subgroup $S$. As $K$ has finite index in $W_\Delta$, it follows that $\overline H' = \overline H \cap K$ contains a 
	finite-index subgroup of $S$. As any finite-index subgroup of $S$ is also isomorphic to a closed surface group, it follows that $\overline H'$ is not free. 
	It remains to show that $\overline H'$ is normal in $K$.
	
	Let $\Omega$ be the completion of $H$ shown in Figure~2 of \cite{DL}.
	 The generalization $\overline \Omega$ of $\Omega$ is a completion of $\overline H$ by Lemma~\ref{lem:completion}. 	
	The $1$-skeleton of the complex $\Omega$ is isomorphic to an infinite bipartite graph.
	 Moreover, it is readily seen that given two vertices $u$ and $v$ (and corresponding vertices $\overline u$ and $\overline v$ in $\overline \Omega$) in a common 
		part of the bipartition of this graph,
	  there is a label preserving automorphism of $\Omega$ (resp. $\overline \Omega$) sending $u$ to $v$ (resp. $\overline u$ to $\overline v$).

	To see that $\overline H'$ is normal in $K$, let $k$ be a reduced word representing an element of $K$ and $h$ be a reduced word representing an element of $\overline H'$. It is enough to show that $khk^{-1} \in \overline H$. 
	Let $\gamma$ be a path in $\overline \Omega$ labeled by $khk^{-1}$ and based at the base vertex
$b$.
	 Such a path exists, since
	for every $s \in V(\Delta)$ and for every vertex $v$ of $\overline \Omega$, there is an edge labeled by $s$ incident to $v$. 
	Let $u$ be the endpoint of the initial subpath subpath of $\gamma$ labeled by $k$. As $k$ has even length, $b$ and $u$ are in the same part of the bipartition of $\overline \Omega$ and there is an automorphism of $\overline \Omega$ taking $b$ to $u$. In particular, as there is a loop in $\overline \Omega$ based at $b$ with label $h$ (by Theorem~\ref{thm:completions}(\ref{thm:completions_loops})), there is also a loop based at $v$ with label $h$. From this and as $\overline \Omega$ is folded, we conclude that $\gamma$ is a loop.
	Thus, again by Theorem~\ref{thm:completions}(\ref{thm:completions_loops}), $khk^{-1}$ represents an element of $\overline H$.
	We have thus shown that $\overline H'$ is normal in $K$, concluding our proof.
\end{proof}

\bibliographystyle{amsalpha}
\bibliography{refs}

\end{document}